\documentclass[11pt]{amsart}
\usepackage{amssymb, enumerate, mdwlist, amsthm, amsmath, bm}
 \usepackage{graphicx, xcolor, hyperref}
\hypersetup{
    colorlinks=true,
    citecolor=blue,
    linkcolor=red,
    urlcolor=orange,
}

\evensidemargin\paperwidth
\advance\evensidemargin-\textwidth
\advance\oddsidemargin-1in
\evensidemargin\oddsidemargin

\topmargin\paperheight
\advance\topmargin-\textheight
\advance\topmargin-1in

\def\Sk{{\mathcal S_{\ast_{k}}}}
\def\Skm{{\mathcal S_{\ast_{k-1}}}}
\def\XX{{\tilde X}}
\def\aa{{\bm a}}

\theoremstyle{plain}
\newtheorem{theorem}{Theorem}[section]
\theoremstyle{plain}

\theoremstyle{plain}
\newtheorem{lemma}[theorem]{Lemma}
\theoremstyle{plain}

\theoremstyle{plain}
\newtheorem{problem}[theorem]{Problem}
\theoremstyle{plain}

\theoremstyle{definition}

\theoremstyle{remark}

\theoremstyle{remark}

\theoremstyle{remark}

\title{Avoiding  a star of three-term arthmetic progressions}

\author{Masato Mimura\and Norihide Tokushige}

\address{Masato Mimura\\
Mathematical Institute, Tohoku University, Japan}
\email{m.masato.mimura.m@tohoku.ac.jp}
\address{Norihide Tokushige\\
College of Education, Ryukyu University, Japan}
\email{hide@u-ryukyu.ac.jp}

\date{\today}

\begin{document}

\begin{abstract}
We provide an upper bound of the size of a subset $A$ of $\mathbb{F}_p^n$ that does not admit a  $k$-star of $3$-APs. Namely, $A$ contains no configuration of $k$ $3$-APs, sharing the middle term, such that all $2k+1$ terms are distinct. In the proof, we adapt a new method in the recent work of Sauermann. 
\end{abstract}

\maketitle

\section{Introduction}\label{section=introduction}

Let $p$ be a prime number at least $3$. Inspired by the breakthrough by Croot, Lev and Pach \cite{CrootLevPach}, Ellenberg and Gijswijt \cite{EllenbergGijswijt} succeeded in providing an upper bound on the number of elements of a set $A$ in $\mathbb{F}_p^n$ that is (\textit{non-degenerate}) \textit{$3$-AP-free}. That means, there exists no triple $(x_1,x_2,x_3)=(a,b,c)\in A$ consisting of three \textit{distinct} elements that is a solution of the linear equation $x_1-2x_2+x_3=0$. (`AP' stands for an arithmetic progression.) Their result states that such an $A$ satisfies 
\[\label{spade}
\#A \leq (\Lambda_{1,\frac{1}{3},p-1})^n\tag{$\spadesuit$}
\]
for all $p\geq 3$ and $n\geq 1$, where $\#A$ means the cardinality of $A$. 
Here for $\alpha\in \mathbb{R}_{>0}$ and for $m,h\in \mathbb{N}$, $\Lambda_{m,\alpha,h}$ denotes the minimum of the function $G\colon (0,1]\to \mathbb{R}$; $G(u)=u^{-\alpha h}(1+u+u^2+\cdots +u^{mh})$. (It is more common to use the symbol $\Gamma_{p,3}$ for this constant; we use the symbol above to follow more general notation in our sequel paper \cite[Definition~2.6]{MimuraTokushige}.) It follows from \cite[Proposition~4.12]{BCCGNSU} that $0.8414 p\leq \Lambda_{1,\frac{1}{3},p-1}\leq 0.918p$ for all $p\geq 3$. 

In this short paper, instead of the equation $x_1-2x_2+x_3=0$, we study the following system of linear equations in $\mathbb{F}_p^n$:
\begin{equation*}\label{equation:star}
\left\{
\begin{aligned}
x_1+x_2-2x_{2k+1}&=0,\\
x_3+x_4-2x_{2k+1}&=0,\\
x_5+x_6-2x_{2k+1}&=0,\\
\vdots\qquad\quad\\
x_{2k-1}+x_{2k}-2x_{2k+1}&=0.
\end{aligned}
\right.
\tag{$\Sk$}
\end{equation*}

The configuration of the solutions of $\Sk$ 
is a \emph{$k$-tuple of `crossing $3$-APs'}, namely, $k$ 
(not necessarily distinct) $3$-APs 
$(x_1,x_{2k+1},x_2)$, $(x_3,x_{2k+1},x_4),\ldots$, and $(x_{2k-1},x_{2k},x_{2k+1})$, with common middle term $x_{2k+1}$. 
For $A\subseteq \mathbb{F}_p^n$, an $\Sk$-\emph{semishape} in $A$ means 
a solution $(x_1,x_2,\ldots,x_{2k+1})\in A^{2k+1}$ of $\Sk$. 
(Semishapes correspond to \emph{cycles} in \cite{Sauermann}.) 
We say it is \emph{degenerate} if $\#\{x_1,x_2,\ldots,x_{2k+1}\}<2k+1$; 
otherwise, an $\Sk$-semishape is called an $\Sk$-\emph{shape}, or simply
a \textit{$k$-star}. 
Namely, an $\Sk$-shape in $A$ is a `distinct $k$-star of $3$-APs' 
in $A$. See figure~\ref{fig1}.

\begin{center}
\begin{figure}[h]
\includegraphics[width=5.5cm]{./fig.8}
\caption{A $k$-star ($\Sk$-shape)}\label{fig1}
\end{figure}
\end{center}

The goal of the paper is to provide an upper bound on the size of $A\subseteq \mathbb{F}_p^n$ that is \emph{$\Sk$-shape-free}. Our main theorem states that \emph{essentially the same estimate} as \eqref{spade}, up to scalar multiple by $k^2$, applies to our setting.

\begin{theorem}[Avoiding a $k$-star]\label{theorem=crossing}
Let $p\geq 3$ be a prime and $n,k\geq 1$. Assume that $A\subseteq \mathbb{F}_p^n$ does not admit any $\Sk$-shape. Then we have 
\[\label{club}
\#A\leq k^2(\Lambda_{1,\frac{1}{3},p-1})^n.\tag{$\clubsuit$}
\]
\end{theorem}

We emphasize that the following two types of problems are completely different in general: the problem of \emph{prohibiting semishapes that are not singletons}, and that of \emph{avoiding shapes}. Here a semishape is called a \emph{singleton semishape} if it consists of one point. For instance, it is rather `straightforward' to observe that the same bound as \eqref{spade} applies if we furthermore require that the set $A\subseteq \mathbb{F}_p^n$ does not admit non-singleton $\Sk$-semishapes. Indeed, if $(a,b,c)$ is a non-degenerate $3$-AP in $A$, then $(a,c,b,b,\ldots,b)$ is a non-singleton $\Sk$-semishape. This argument \emph{breaks down} in our setting. Tao's \emph{slice rank method} \cite{Tao}, which reformulated the proof technique in \cite{EllenbergGijswijt}, works effectively for problems of the former type. In contrast, there is no direct interpretation in terms of slice rank for problems of the latter type. Compare with \cite[Definition~2.2]{MimuraTokushige}.

Naslund addressed \cite{Naslund} the problem of the latter type for the equation $x_1+x_2+\cdots +x_p=0$ in $\mathbb{F}_p^n$. Recently, Sauermann \cite{Sauermann} has provided substantial improvement of the upper bound in that setting. Our proof of Theorem~\ref{theorem=crossing} is inspired by Sauermann's method. We note that there are two major differences from her original argument. First, one of the keys \cite[Lemma~3.1]{Sauermann} to her proof collapses when we collect $\Sk$-semishapes of the form $(a,c,b,b,\ldots,b)$, which appeared in the argument above. This issue arises from a lack of the full symmetry of the coefficients in the system $\Sk$ of equations. We overcome this difficulty by \emph{choosing a `well-behaved' labeling of points} in semishapes from a fixed point-configuration type; see the argument in `Case~2' in the proof of Theorem~\ref{theorem=crossing}. Secondly, in our modification (Lemma~\ref{lemma=extendability}) of the key lemma \cite[Lemma~3.1]{Sauermann}, works `too' ideally. It enables us to have a considerably better bound in Theorem~\ref{theorem=crossing}: the direct application to Sauermann's method would give a bound of the form $\#A \leq C \left(\sqrt{\Lambda_{1,\frac{1}{3},p-1}\cdot p}\right)^n$. 

\section{Proof of Theorem~$\ref{theorem=crossing}$}\label{section=proof}
For $X_1,X_2,\ldots ,X_{2k+1}\subseteq  \mathbb{F}_p^n$, an $\Sk$-semishape in $X_1\times X_2\times \cdots \times X_{2k+1}=:\XX$ means $(x_1,x_2,\ldots,x_{2k+1})\in\XX$ that is a solution of $\Sk$. Two $\Sk$-semishapes $(x_1,x_2,\ldots,x_{2k+1})$ and $(x'_1,x'_2,\ldots,x'_{2k+1})$ are said to be \emph{disjoint} if $\{x_1,x_2,\ldots,x_{2k+1}\}$ and $\{x'_1,x'_2,\ldots,x'_{2k+1}\}$ are disjoint. We define disjointness of $3$-APs in the same manner as above. Following \cite{Sauermann}, we define the following notion: for $1\leq i<j\leq 2k+1$, we say that $(x,y)\in X_i\times X_j$ is \emph{$(i,j)$-extendable} in $\XX$ if there exists an $\Sk$-semishape in $\XX$ whose $i$-th term is $x$ and $j$-th term is $y$. Here for a semishape $(x_1,x_2,\ldots,x_{2k+1})$, the $i$-th term of it means $x_i$. 

In the proof of Theorem~\ref{theorem=crossing}, we employ the following result, which may be seen as the \textit{multicolored version} of the bound of `non-singleton $\Sk$-semishape'-free subsets. See \cite[Section~$3$]{BCCGNSU} for more backgrounds on the multicolored version. This might be known to experts; at least, the statement for $k=1$ is a well-known generalization of the Ellenberg--Gijswijt theorem \cite{EllenbergGijswijt}. Nevertheless, we include an outlined proof for the reader's convenience.

\begin{theorem}[Multicolored $\Sk$-free set; compare with \cite{EllenbergGijswijt}]\label{theorem=multicolor}
Let $s\in \mathbb{N}$ and let
\[
 M=\{(x_{1,i},x_{2,i}\ldots,x_{2k+1,i})\in(\mathbb{F}_p^n)^{2k+1}:
1\leq i\leq s\}.
\]
Assume that they satisfy that 
\[
(x_{1,i_1},x_{2,i_2},\ldots,x_{2k+1,i_{2k+1}})\textrm{ is an $\Sk$-semishape}\quad \Longleftrightarrow \quad i_1=i_2=\cdots=i_{2k+1}.
\]
Then, we have $s\leq (\Lambda_{1,\frac{1}{3},p-1})^n$.
\end{theorem}

\begin{proof}
Here we only provide an outlined proof. For $k=1$, as we mentioned above, the assertion is well known to experts. It follows from the proof of the Ellenberg--Gijswijt bound \cite{EllenbergGijswijt} by slice rank method. See \cite[Section~$3$]{BCCGNSU}; we also refer the reader to \cite[Definition~4.2 and Corollary~C]{MimuraTokushige} and \cite[Section~9]{LovaszSauermann} for more details on the use of the slice rank method in the multicolored setting.

For $k\geq 2$, we may reduce the problem to the case where $k=1$ in the following argument; it is similar to the `straightforward' argument mentioned in the Introduction from $3$-APs to $\Sk$-semishapes. Fix $k\geq 2$. Suppose that  $s> (\Lambda_{1,\frac{1}{3},p-1})^n$. Apply the assertion of Theorem~\ref{theorem=multicolor} for $k=1$ to $X_1\times X_2\times X_{2k+1}$. Here for $j\in [2k+1]$, we define  $X_j:=\{x_{j,i}:1\leq i\leq s \}$. It implies that there exist $i_1,i_2,i_{2k+1}$ with $\#\{i_1,i_2,i_{2k+1}\}\geq 2$ such that $(x_{1,i_1},x_{2,i_2},x_{2k+1,i_{2k+1}})$ is a $3$-AP. Then, the following $(2k+1)$-tuple
\[
(x_{1,i_1},x_{2,i_2},x_{3,i_{2k+1}},x_{4,i_{2k+1}},\ldots,x_{2k,i_{2k+1}},x_{2k+1,i_{2k+1}})
\]
is an $\Sk$-semishape in $X_1\times X_2\times \cdots \times X_{2k+1}$; it violates the condition above for $k$. Therefore, we obtain the bound $s\leq (\Lambda_{1,\frac{1}{3},p-1})^n$ even for $k\geq 2$.
\end{proof}

\begin{proof}[Proof of Theorem~$\ref{theorem=crossing}$]
We prove the statement of the theorem by induction on $k$. 
The initial case $k=1$ is the well-known result of Ellenberg and Gijswijt
\cite{EllenbergGijswijt}. 

So we move on to the induction step. Let $k\geq 2$.
We assume that the statement holds for $k-1$, and we shall show the statement 
for $k$. (A reader may read the proof below assuming $k=3$, which
essentially contains everything needed for the general case.) In the induction step, we consider both $\Sk$-(semi)shapes and $\Skm$-(semi)shapes. Since (semi)shapes of our main concern are $\Sk$-(semi)shapes, we call them simply  (semi)shapes. Hence, a `(semi)shape' in the remaining part of the proof always means an $\Sk$-(semi)shape; we do not use any abbreviation for an $\Skm$-(semi)shape.

Let $A\subseteq \mathbb{F}_p^n$ be shape-free. 
Let $t:=\left\lceil\frac{1}{k^2}\#A \right\rceil$. 
We distinguish the following two cases.
\begin{itemize}
\item[\emph{Case~$1$.}] There does not exist $t$ disjoint $\Skm$-shapes in $A$.
\item[\emph{Case~$2$.}] There \emph{does} exist $t$ disjoint $\Skm$-shapes in $A$.
\end{itemize}

In Case~$1$, take a maximal family of disjoint $\Skm$-shapes in $A$, and delete all the points in these $\Skm$-shapes from $A$. 
Note that each deleted $\Skm$-shape has $2k-1$ points.
So the resulting set $A'$ satisfies
\[
\#A'\geq \#A-(2k-1)(t-1)\geq \left(\frac{k-1}k\right)^2\#A.  
\]
By construction, this $A'\subseteq \mathbb{F}_p^n$ is $\Skm$-shape-free, 
and it follows from the induction hypothesis that $\#A'\leq (k-1)^2(\Lambda_{1,\frac{1}{3},p-1})^n$. Hence, in Case~$1$, we obtain that 
\[
\#A\leq k^2(\Lambda_{1,\frac{1}{3},p-1})^n.
\]

In Case~$2$, take $t$ disjoint $\Skm$-shapes in A. 
For each $\Skm$-shape $\aa:=(a_1,a_2,\dots,a_{2k-1})$, where the last element 
$a_{2k-1}$ is the common middle term of the 3-APs, we see that
\[
\tilde\aa:=(a_1,a_2,a_1,a_2,a_3,\ldots,a_{2k-1})\in A^{2k+1}, 
\]
which is a concatenation of $(a_1,a_2)$ and $\aa$,
is a semishape in $A$. Define a collection $M$ of $t$ disjoint 
semishapes 
\[
M:=\{\tilde\aa\in A^{2k+1}: \aa\in A^{2k-1} 
\textrm{ is in the chosen $t$ disjoint $\Skm$-shapes}\}.
\]
For each $i\in [2k+1]$, let $X_i$ as the set of all $i$-th term of semishapes in $M$. By disjointness, $\#X_i=t$ for each $i\in [2k+1]$. 
Furthermore, $2k-1$ sets $X_1$, $X_3$, $X_5,X_6,\ldots,X_{2k+1}$ are pairwise disjoint; $X_1=X_3$ and $X_2=X_4$. We claim that every semishape 
$(x_1,x_2,\ldots,x_{2k+1})$ in $\XX:=X_1\times X_2\times \cdots \times X_{2k+1}$ satisfies that $x_1=x_3$ and that $x_2=x_4$. Indeed, if $x_1\ne x_3$, then $x_1+x_2=x_3+x_4(=2x_{2k+1})$ implies that $x_2\neq x_4$.
Using pairwise disjointness of $X_1,X_2,X_5,X_6,\ldots ,X_{2k+1}$, we have that
$x_1,x_2,\ldots,x_{2k+2}$ are all distinct $2k+1$ points which consist
of a shape in $A$, a contradiction.
Similarly, we have that $x_2=x_4$. Now define
\[
B:=\{(x,y)\in X_1\times X_{2k+1}:\textrm{ $(x,y)$ is  $(1,2k+1)$-extendable in 
$\XX$}\}.
\]
The following is the key lemma to the proof, which substitutes for \cite[Lemma~3.1]{Sauermann}:

\begin{lemma}\label{lemma=extendability}
Let $(x,y)\ne (x',y')$ be two elements in  $B$.  Then we have that $y\ne y'$.
\end{lemma}

\begin{proof}
Suppose the contrary: $(x,y)\ne (x',y')$ and $y=y'$. Then $x\ne x'$. By assumption, there must exist two distinct semishapes 
\[
(x,x_2,x,x_2,x_5,x_6\ldots,x_{2k},y) \text{ and }
(x',x'_2,x',x'_2,x'_5,x'_6\ldots,x'_{2k},y) 
\]
in $\XX$.
However, then we would find a semishape 
$(x,x_2,x',x'_2,x_5,x_6\ldots,x_{2k},y)$ in $\XX$, 
which violates $x_1=x_3$. A contradiction.
\end{proof}

Note that Lemma~\ref{lemma=extendability} has a great byproduct: It implies that $B\ni (x,y)\mapsto y\in X_{2k+1}$ is \emph{bijective}; in particular, it follows that $\#B=t$. This explains the second difference from Sauermann's original argument, as we discussed in the Introduction. (In her setting, the outcome is that $\#B\leq p^n$.) 
In what follows, we claim that every semishape in $\XX$ is in fact an element in $M$, in other words, we verify that we can apply 
Theorem~\ref{theorem=multicolor} to $M$.
To show this, we change the indices and express $M$ as 
\[
M=\{(x_{1,i},x_{2,i},\ldots,x_{2k+1,i}):1\leq i\leq t\}.  
\]
Assume that $\bm x:=(x_{1,i_1},x_{2,i_2},\ldots ,x_{2k+1,i_{2k+1}})$ is a semishape in $\XX$. Then, it is easy to see by disjointness that $i_1=i_3$ and $i_2=i_4$; by our observation on $(1,2k+1)$-extendable pairs, it also follows that $i_1=i_{2k+1}$. These equalities imply that $i_1=i_2=i_3=i_4=i_{2k+1}=:i$. Indeed, note that $x_{1,i_1}+x_{2,i_2}-2x_{2k+1,i_{2k+1}}=0$ and $x_{3,i_1}+x_{4,i_2}-2x_{2k+1,i_{2k+1}}=0$. 
What remains is to show that $i_5=i_6=\cdots=i_{2k}=i$. 
Assume the contrary, say, $i_5\neq i$. Then it follows from
$x_{5,i}+x_{6,i}=x_{5,i_5}+x_{6,i_6}(=2x_{2k+1,i})$ that $i_6\neq i$.
Consider the $(2k+1)$-tuple obtained from $\bm x$ by
replacing $x_{3,i}, x_{4,i}$ with $x_{5,i_5},x_{6,i_6}$, that is, 
\[
(x_{1,i},x_{2,i},x_{5,i_5},x_{6,i_6},x_{5,i},x_{6,i},x_{7,i},x_{8,i}
\ldots,x_{2k+1,i}). 
\]
It, however, must be a shape in $A$ by disjointness; this contradicts the choice of $A$. Hence we prove the claim.

To close up the argument for Case~$2$, apply Theorem~\ref{theorem=multicolor} to $M$. We obtain that 
\[
t=\#M\leq (\Lambda_{1,\frac{1}{3},p-1})^n.
\]
Recall that $t=\left\lceil\frac{1}{k^2}\#A\right\rceil$. Hence, in Case~$2$, we conclude that 
\[
\#A\leq k^2(\Lambda_{1,\frac{1}{3},p-1})^n.
\]

Unifying both cases yields \eqref{club}.
\end{proof}

\section{Further direction}
In Theorem~\ref{theorem=crossing} we considered 3-APs sharing the same middle
term. 
We may relax this condition a little bit. More precisely, the statement of 
Theorem~\ref{theorem=crossing} still holds if we replace $\Sk$
with the following system of equations:
\[
\left\{
\begin{aligned}
x_1-2x_2+x_{2k+1}&=0,\\
x_3-2x_4+x_{2k+1}&=0,\\
\vdots\qquad\quad\\
x_{2k-1}-2x_{2k}+x_{2k+1}&=0.
\end{aligned}
\right.
\]
The reader is invited to verify that the
proof of Theorem~\ref{theorem=crossing} is easily extended to the case above. However, the following problem is already beyond the reach of the proof
 of Theorem~\ref{theorem=crossing}.

 \begin{problem}
 Let $A\subseteq \mathbb{F}_p^n$ do not admit any shape defined by
 \[
 \left\{
 \begin{aligned}
 x_1-2x_2+x_{5}&=0,\\
 x_3+x_4-2x_{5}&=0.
 \end{aligned}
 \right.
 \]
 Then is it true
 $\# A\leq C(p)^n$ for some constant $C(p)<p$ depending only on $p$?
 \end{problem}

If two variables appear
more than once in the system, then the situation becomes much involved.
In our sequel paper \cite{MimuraTokushige}, we also investigate semishapes associated with the following system of inequalities:
\[
\left\{\begin{aligned}
x_1-x_2-x_3+x_4&=0,\\
x_1-2x_3+x_5&=0.
\end{aligned}
\right.
\tag{$\mathcal{S}_{\mathrm{W}}$}
\]
Geometrically, an $\mathcal{S}_{\mathrm{W}}$-shape $(x_1,x_2,x_3,x_4,x_5)$ is the combination shape of  a `parallelogram $(x_1,x_2,x_4,x_3)$' and a `$3$-AP $(x_1,x_3,x_5)$'; it may be regarded as a `\textit{$\mathrm{W}$ shape}'. See figure~\ref{fig2}. Here we announce the following result in \cite[Theorem~1.1 and Theorem~B]{MimuraTokushige}:

\begin{center}
\begin{figure}[h]
\includegraphics[width=5.0cm]{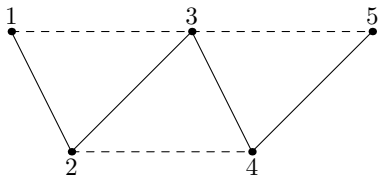}  
\caption{An $\mathcal{S}_{\mathrm{W}}$-shape}\label{fig2}
\end{figure}
\end{center}

\begin{theorem}[Avoiding a `$\mathrm{W}$ shape']\label{theorem=Wshape}
Let $p\geq 3$ be a prime. Then there exists $C_{\mathrm{W}}(p)<p$ that satisfies the following:  For $n\geq 1$, assume that $A\subseteq \mathbb{F}_p^n$ does not admit any $\mathcal{S}_{\mathrm{W}}$-shape. Then we have 
\[
\#A\leq 7\left(\sqrt{C_{\mathrm{W}}(p)\cdot p}\right)^n.
\]
Furthermore, we have that 
\[
C_{\mathrm{W}}(p)\leq \inf\{\max\{\Lambda_{1,\alpha,p-1},\Lambda_{2,\beta,p-1}\}: 3\alpha+2\beta=1,\ \alpha,\beta>0\}. 
\]
\end{theorem}
For sufficiently large $p$, we have that $C_{\mathrm{W}}(p)\leq (0.985)^2p$.
See \cite[Remark~4.3]{MimuraTokushige} for details on this estimate.

\section*{Acknowledgments}
The authors are grateful to the members, Wataru Kai, Akihiro Munemasa, Shin-ichiro, Seki and Kiyoto Yoshino,  of the ongoing seminar on the Green--Tao theorem (on arithmetic progressions in the set of prime numbers) at Tohoku University launched in  October, 2018. Thanks to this seminar, the first-named author has been intrigued with the subject of this paper. They thank Takumi Yokota for comments. 
Masato Mimura is supported in part by JSPS KAKENHI Grant Number JP17H04822, and
Norihide Tokushige is supported by JSPS KAKENHI Grant Number JP18K03399.

\bibliographystyle{amsalpha}
\bibliography{kstar.bib}
\end{document}